   \documentclass[twoside,reqno,11pt]{fcaa-var} %

\usepackage{graphicx}
\usepackage{epsfig}

\usepackage{amsthm}
\usepackage{amsmath}
\usepackage{latexsym}
\usepackage{amsfonts}
\usepackage{amssymb}

\newtheoremstyle{theorem}
  {15pt}          
  {15pt}  
  {\sl}  
  {\parindent}
  {\sc}  
  {. }   
  { }    
  {}     
\theoremstyle{theorem}

\newtheorem{theorem}{Theorem}[section]

\newtheoremstyle{defi}
  {15pt}          
  {15pt}  
  {\rm}  
  {\parindent}     
  {\sc}  
  {. }    
  { }    
  {}     
\theoremstyle{defi}
\newtheorem{definition}{Definition}[section]
\newtheorem{remark}{Remark}[section]

 \def\proofname{\indent\rm P\ r\ o\ o\ f.}
 
 \def\qed{\hfill$\Box$} 

\numberwithin{equation}{section}
\theoremstyle{plain}

\theoremstyle{definition}
\newtheorem{problem}[theorem]{Problem}
\newtheorem{assumption}[theorem]{Assumption}
\def\L{{\mathcal L}}
\def\p#1{{\left({#1}\right)}}

\def\R{\mathcal R}
\def\G{{\mathbb G}}

 \usepackage{hyperref} 

 \def\theequation{\arabic{section}.\arabic{equation}}

 \def\themycopyrightfootnote{\vspace*{3pt}
 \copyright \, Year\,  Diogenes  Co., Sofia
   \par  \noindent pp. xxx--xxx, DOI: ......................
   \hfill  \vspace*{-36pt}
  }

 \title[ON A NON--LOCAL PROBLEM \dots]{ON A NON--LOCAL PROBLEM \\ FOR A MULTI--TERM FRACTIONAL \\  DIFFUSION-WAVE EQUATION}
 \author[\normalsize M. Ruzhansky, N. Tokmagambetov, B. T. Torebek]{\normalsize Michael Ruzhansky $^{1,2}$, Niyaz Tokmagambetov $^{1,3,4}$, \\ and Berikbol T. Torebek $^{1,3,4}$}

\begin{document}

 \vbox to 1.5cm { \vfill }


 \bigskip \medskip

 \begin{abstract}
This paper deals with the multi-term generalisation of the time-fractional diffusion-wave equation for general operators with discrete spectrum, as well as for positive hypoelliptic operators, with homogeneous multi-point time-nonlocal conditions. Several examples of the settings where our nonlocal problems are applicable are given. The results for the discrete spectrum are also applied to treat the case of general homogeneous hypoelliptic left-invariant differential operators on general graded Lie groups, by using the representation theory of the group.
For all these problems, we show the existence, uniqueness, and the explicit representation formulae for the solutions.
 \medskip


 \smallskip


 \end{abstract}

 \maketitle

 \vspace*{-20pt}



\section{Introduction}

A time-fractional diffusion-wave equation is a partial integro--differential equation obtained from the classical heat or wave equation by replacing the first or second order time-fractional derivative \cite{M97}. Analysis of diffusion-wave equations with time-fractional derivative, which are obtained from the classical heat and wave equations of mathematical physics by replacing the time derivative by a fractional derivative, constitute a field of growing interest, as is evident from literature surveys. Lots of universal phenomena can be modeled to a greater degree of accuracy by using the properties of these evolution equations.

In \cite{OS74} the authors studied a diffusion equation with time-fractional derivative that contains derivative of first order in space and half order derivative in time variable. They also discussed the relationship between their fractional diffusion equation and a classical diffusion equation. Nigmatullin \cite{N86} pointed out that many of the universal electromagnetic, acoustic, and mechanical responses can be modeled accurately using the diffusion-wave equations with time-fractional derivative. It was shown that the found memory function in one case describes pure diffusion (absence of memory), and in another case turns into the wave equation (full memory). In \cite{A02, A03} Agrawal analysed the fractional evolution equation on a half-line and in a bounded domain. The time-fractional one-dimensional diffusion-wave equation was studied by Gorenflo and Mainardi in \cite{GM98}.

The generalised diffusion equation with time-fractional derivative corresponds to a continuous time random walk model where the characteristic waiting time elapsing between two successive jumps diverge, but the jump length variance remains finite and is proportional to $t^\alpha.$ The exponent $\alpha$ of the mean square displacement proportional to $t^\alpha$ often does not remain constant and changes. To adequately describe these phenomena with fractional models, several approaches have been suggested in the literature (for example, \cite{W86, SW89, CLA08, GM09, L09, TT18c}). However, studies of the generalised time-fractional partial differential equations with three kinds of nonhomogeneous boundary conditions are still limited.

The time-fractional diffusion-wave equation can be written as
\begin{equation}\label{I_1}\partial^\alpha_{+0,t}u(t,x)=k u_{xx}(t,x)+f(t,x),\,t>0,\,0<x<L,\end{equation}
where $x$ and $t$ are the space and time variables, $k$ is an arbitrary positive constant, $f(t,x)$ is a sufficiently smooth function, $0 < \alpha \leq 2$ and $\partial^\alpha_{+0,t}$ is a Caputo fractional derivative of order $\alpha$ defined as Def. \ref{def3}.

When $1 < \alpha < 2,$ equation \eqref{I_1} is the time-fractional wave equation and when $0 < \alpha < 1,$ equation \eqref{I_1} is the time-fractional diffusion equation. When $\alpha = 2,$ it represents a classical wave equation; while if $\alpha = 1,$ it represents a classical diffusion equation.

In some practical situations the underlying processes cannot be described by the equation \eqref{I_1}, but can be modeled using its generalisation the multi-term time-fractional equation that are given by \cite{L11}, namely
\begin{multline}\label{I_2}\partial^{\alpha}_{+0,t}u(t,x)-\sum\limits_{j=1}^m a_j\partial^{\alpha_j}_{+0,t}u(t,x)\\=ku_{xx}(t,x)+f(t,x), \,\, t>0,\,0<x<L,\end{multline} where $0<\alpha_j\leq 1,\, \alpha_j\leq \alpha\leq 2,\, a_j\in \mathbb{R},\, m\in \mathbb{N}.$

Investigation of different initial-value problems for the equation \eqref{I_2} was performed in \cite{L11, JLTB12, LLY15, L17, KMR18}. We also refer some papers on these topics \cite{AKMR17, KMR18, KT19, Yam1, Yam2, Zacher}. In \cite{KMR18} the authors dealt with the non-local initial boundary problem for multi-term time-fractional PDE with Bessel operator. They used Fourier-Bessel series expansion in order to find the explicit solution for the considered problem, yielding also its existence. We also note that, in \cite{RuzTT20} the authors of this paper investigated a time-nonlocal problem for the integro-differential diffusion-wave equation on the Heisenberg group.

Certain types of physical problems can be modelled by heat and wave equations with multi-point (nonlocal) initial conditions. The heat and wave equations with time multi-point conditions can arise from studying the atomic reactors \cite{By2, By3}. Well-posedness and numerical simulations of heat and wave equations with time nonlocal conditions were studied in \cite{By1, Chab, Deh1, Deh2}.

In the present paper we deal with multi-term generalisations of the fractional diffusion-wave equation \eqref{I_2} in two settings:

\begin{itemize}
\item {\bf Operators with discrete spectrum:} Let $D\subset \mathbb{R}^d,\,d\geq 1,$ be a bounded domain with smooth boundary $S$, or an unbounded domain. In a cylindrical domain $\Omega=\{(t,x):\,(0,T)\times D\}$ we consider the equation
\begin{equation}\label{1a}
\partial^{\alpha}_{+0,t}u(t,x)-\sum\limits_{j=1}^m a_j\partial^{\alpha_j}_{+0,t}u(t,x)+\mathcal{L} u(t,x)=f(t,x),
\end{equation}
where $\partial^{\alpha}_{+0,t}$ is a Caputo fractional derivative, $0<\alpha_j\leq 1,\, \alpha_j\leq \alpha\leq 2,\, a_j\in \mathbb{R},\, m\in \mathbb{N},$ $f$ is a given function. The main assumption here is that $\mathcal{L}$ is a linear self-adjoint operator with a discrete spectrum (on a separable Hilbert space $\mathcal{H}$).
This setting will be considered in Section \ref{SEC:opL}. There, we also review a number of problems that are covered by this analysis:
\begin{itemize}
\item Sturm-Liouville problem;
\item differential operators with involution;
\item Bessel operator;
\item fractional Sturm-Liouville operator;
\item harmonic oscillator;
\item anharmonic oscillator;
\item Landau Hamiltonian.
\end{itemize}
\item {\bf Hypoelliptic differential operators (Operators with continuous spectrum):} According to the Rothschild-Stein lifting theorem (\cite{RS76}) and its further developments, a model case for general hypoelliptic differential operators on manifolds are the so-called Rockland operators on graded Lie groups. Thus, let $\mathbb{G}$ be a graded Lie group and let $\mathcal{R}$ be a positive left-invariant homogeneous hypoelliptic operator on $\mathbb{G}$. Then, in the set  $\Omega=\{(t,x):\,(0,T)\times \mathbb{G}\}$,  we consider the equation
\begin{equation}\label{iR: 1}
\partial^{\alpha}_{+0,t}u(t,x)-\sum\limits_{j=1}^m a_j\partial^{\alpha_j}_{+0,t}u(t,x)+\mathcal{R} u(t,x)=f(t,x), \end{equation}
with non-local initial conditions
\begin{equation*}\label{iR: 2}
u(0,x)-\sum\limits_{i=1}^n\mu_i u(T_i,x)=0,\, [\alpha]u_t(0,x)=0,\,x\in \mathbb{G},
\end{equation*}
where $\partial^{\alpha}_{+0,t}$ is a Caputo fractional derivative, $0<\alpha_j\leq 1,\, \alpha_j\leq \alpha\leq 2,\, a_j\in \mathbb{R},\, m\in \mathbb{N}$, and $\mu_i\in \mathbb{R},\,0<T_1\leq T_2\leq ...\leq T_n=T.$ This setting will be analysed in Section \ref{SEC:hyp}. There, we also review a number of problems that are covered by this analysis:
\begin{itemize}
\item Higher order elliptic differential operators;
\item sub-Laplacian on the Heisenberg group;
\item sub-Laplacian on the Carnot group.
\end{itemize}
\end{itemize}

We note that the spectrum of the operator $\mathcal{R}$ on $L^2(\mathbb{G})$ is continuous, however, if $\pi\in\widehat{\mathbb{G}}$ is a representation of $\mathbb{G}$ from the unitary dual $\widehat{\mathbb{G}}$, then the operator symbol $\pi(\mathcal{R})$ of $\mathcal{R}$ at $\pi$ has a discrete spectrum. Therefore, the results obtained in Section \ref{SEC:opL} for \eqref{1a} become applicable with $\mathcal{L}=\pi(\mathcal{R})$. Consequently, we obtain the solution to \eqref{iR: 1} by integrating this solution with respect to the Plancherel measure.

\section{Preliminaries}
\label{Properties}

We start by briefly recalling several notions important for the analysis of this paper.

\subsection{Definitions and properties of fractional operators}
Here, we recall definitions and properties of fractional integration and differentiation operators \cite{SKM87, LG99, N03, KST06}.

First of all, we start by defining the function spaces $C_{\gamma}^{m}[a, b]$, $\gamma\geq-1$, $m\in\mathbb N_{0}:=\mathbb N\cup\{0\}$, $a<\infty$, $b\leq\infty$.
\begin{definition} \cite{Dim82, LG99}.
\label{FS}
Let us consider the set of complex-valued functions $f$ defined on $[a, b]$. Fix $\gamma\geq-1$.
\begin{itemize}
\item We say that $f\in C_{\gamma}^{0}[a, b]:=C_{\gamma}[a, b]$,  if there is a real number $p>\gamma,$ such that
$f(x)=(x-a)^{p}f_{1}(x)$
with $f_{1}\in C[a, b]$.
\item We say that $f\in C_{\gamma}^{m}[a, b]$, $m\in\mathbb N_{0}$, if and only if $f^{(m)}\in C_{\gamma}[a, b].$
\end{itemize}
From \cite{LG99} it follows that $C_{\gamma}[a, b]$ is a vector space and the set of spaces $C_{\gamma}[a, b]$ is ordered by inclusion according to
$$
C_{\gamma}[a, b]\subset C_{\beta}[a, b] \Leftrightarrow \gamma\geq\beta\geq-1.
$$
For further properties of $C_{\gamma}[a, b]$ we refer to \cite{Dim82} and \cite{LG99}.
\end{definition}

\begin{definition} \cite{LG99, KST06} (Riemann-Liouville integral). Let $f\in C_{\gamma}[a, b]$ and $\gamma\geq-1$. The Riemann--Liouville fractional integral $I_{+a} ^\alpha$ of order $\alpha\in\mathbb R$, $\alpha>0$, is defined as
$$
I_{+a} ^\alpha  f\left( t \right) 
={\rm{
}}\frac{1}{{\Gamma \left( \alpha \right)}}\int\limits_a^t {\left(
{t - s} \right)^{\alpha  - 1} f\left( s \right)} ds,
$$
where 
$\Gamma$ denotes the Euler gamma function. When $\alpha=0$ we put
$$
I_{+a}^{0}f(t):=f(t).
$$

Moreover, $I_{+a}^\alpha: C_{\gamma}[a, b]\rightarrow C_{\alpha+\gamma}[a, b] \subset C_{\gamma}[a, b]$ for $\alpha>0$, $\gamma\geq-1$ (see, \cite{LG99}).
\end{definition}


\begin{definition} \cite{KST06} (Riemann-Liouville derivative).
Let $f\in C_{-1}^{m}[a, b]$, $m\in\mathbb N$. The Riemann--Liouville fractional derivative $D_{+a} ^\alpha$ of order $\alpha>0$, $m-1 < \alpha \leq m,$ is defined as
$$
D_{+a} ^\alpha f \left( t \right) = \frac{{d^m }}{{dt^m }}I_{+a} ^{m - \alpha } f \left( t \right).
$$
Note that when $m-1 < \alpha < m,$ we have
$$
D_{+a} ^\alpha f \left( t \right) = {\rm{}}\frac{1}{{\Gamma \left( m-\alpha \right)}}\frac{d^m}{dt^m}\int\limits_a^t {\left({t - s} \right)^{m-1-\alpha} f\left( s \right)} ds.
$$
\end{definition}

\begin{definition}
\label{def3}
\cite{KST06} (Caputo derivative).
Let $f\in C_{-1}^{m}[a, b]$, $m\in\mathbb N$. The Caputo fractional derivative $\partial_{+a}^\alpha$ of order $\alpha\in\mathbb R$, $m-1<\alpha<m,$ is defined as
$$
\partial_{+a}^\alpha  \left[ f \right]\left( t \right) = I_{+a} ^{m - \alpha } f^{(m)}\left( t \right)={\rm{}}\frac{1}{{\Gamma \left( m-\alpha \right)}}\int\limits_a^t {\left({t - s} \right)^{m-1-\alpha} f^{(m)}\left( s \right)} ds.
$$
When $\alpha=m$, we define
$$
\partial_{+a}^\alpha  \left[ f \right]\left( t \right) : = f^{(m)}\left( t \right).
$$
\end{definition}

Note that in monographs \cite{SKM87, N03, KST06}, the authors studied different types of fractional differentiations and their properties. In what follows we formulate statements of necessary properties of integral and integro--differential operators of the Riemann--Liouville type and fractional Caputo operators.

\subsection{Fourier analysis of linear operators with discrete spectrum}
In this section we recall elements of the Fourier analysis developed in the recent investigations \cite{RT16, RT16a, DRT17}.

Let $\mathcal{L}$ be a linear operator in the separable Hilbert space $\mathcal{H}$ and, introduce
$$
\textrm{Dom}\left(\mathcal{L}^\infty\right):=\bigcap\limits_{k=1}^\infty\textrm{Dom}\left(\mathcal{L}^k\right),
$$
where
$\textrm{Dom}\left(\mathcal{L}^k\right)$
is the domain of the iterated operator $\mathcal{L}^k$:
$$\textrm{Dom}\left(\mathcal{L}^k\right):=\left\{f\in \mathcal{H}:\, \mathcal{L}^jf\in \textrm{Dom}\left(\mathcal{L}\right),\,j=0,1,2,...,k-1\right\}.$$

We give the Fr\'{e}chet topology on $\mathcal{H}^\infty_\mathcal{L}$ by the family of semi-norms
\begin{equation}\label{FA-1}\|\varphi\|_{\mathcal{H}^\infty_\mathcal{L}}:=\max_{j\geq k}\left\|\mathcal{L}^j\varphi\right\|_{\mathcal{H}},\, k\in \mathbb{N}_0,\, \varphi\in \mathcal{H}^\infty_\mathcal{L}.\end{equation}

We call $\mathcal{H}_\mathcal{L}^\infty:=\textrm{Dom}\left(\mathcal{L}^\infty\right)$ the space of test functions generated by $\mathcal{L}.$

Analogously to the $\mathcal{H}$-conjugate operator $\mathcal{L}^*$ (to $\mathcal{L}$), define the space of test functions for $\mathcal{L^*}$: $\mathcal{H}_{\mathcal{L^*}}^\infty:=\textrm{Dom}\left(\left(\mathcal{L^*}\right)^\infty\right).$

Now let us introduce the space of linear continuous functionals on $\mathcal{H}^{\infty}_{\mathcal{L}}$ and denote it by $\mathcal{H}^{-\infty}_{\mathcal{L}^*}:=\mathcal{L}\left(\mathcal{H}^{\infty}_{\mathcal{L}}, \mathbb{C}\right)$. We call the last one the space of $\mathcal{L}^*$-distributions. The continuity can be understood in terms of \eqref{FA-1}. Obviously, an embedding $\psi\in \mathcal{H}^{\infty}_{\mathcal{L}^*}\hookrightarrow \mathcal{H}^{-\infty}_{\mathcal{L}^*}$ is true.

In addition, we require that the system of eigenfunctions $\{e_\xi: \xi \in \mathcal{I}\}$ of $\mathcal{L}$ is a Riesz basis in the separable Hilbert space $\mathcal{H}.$ Then its biorthogonal system $\{e^*_\xi: \xi\in \mathcal{I}\}$ is also a Riesz basis in $\mathcal{H}$ (see \cite{B51, G63}). Here the biorthogonality relations mean $$(e_\xi, e^*_\eta) = \delta_{\xi, \eta},$$ where $\delta_{\xi, \eta}$ is the Kronecker delta. Indeed, the function $e^*_\xi$ is an eigenfunction of $\mathcal{L}^*$ corresponding to the eigenvalue $\bar{\lambda}_\xi$ for all $\xi\in \mathcal{I}.$

Denote by $\mathcal{S}(\mathcal{I})$ the space of rapidly decaying functions from $\mathcal{I}$ to $\mathbb{C}$: $\varphi\in \mathcal{S}(\mathcal{I})$ if for arbitrary $m < \infty$ there is a constant $C_{\varphi,m}$ such that $$|\varphi(\xi)|\leq C_{\varphi,m}\langle\xi\rangle^{-m}$$ is valid for any $\xi\in \mathcal{I},$ where $\langle\xi\rangle:=(1+|\xi|)^{1/2}.$

The family of seminorms $p_k:$
$$
p_k(\varphi) :=\sup\limits_{\xi\in \mathcal{I}}\langle\xi\rangle^k|\varphi(\xi)|,
$$
where $k \in \mathbb{N}_0$, defines the topology on $\mathcal{S}(\mathcal{I})$.

Define the $\mathcal{L}$-Fourier transform as the linear mapping
$$
(\mathcal{F}_\mathcal{L}f)(\xi)=(f\mapsto \hat{f}): H^\infty_\mathcal{L}\rightarrow \mathcal{S}(\mathcal{I})
$$
by the formula $$\hat{f}(\xi):=(\mathcal{F}_\mathcal{L}f)(\xi)=(f,e^*_\xi).$$
Also, in a similar way we introduce the $\mathcal{L}^*$-Fourier transform on $H^\infty_\mathcal{L^*}$ by
$$
(\mathcal{F}_\mathcal{L^{\ast}}g)(\xi)=(g,e_\xi),
$$
for $g\in H^\infty_\mathcal{L^*}.$

The ($\mathcal{L^*}$-)$\mathcal{L}$-Fourier transform $$\mathcal{F}_\mathcal{L}: H^\infty_\mathcal{L}\to\mathcal{S}(\mathcal{I}) (\mathcal{F}_\mathcal{L^*}: H^\infty_\mathcal{L^*}\to\mathcal{S}(\mathcal{I}))$$ is a bijective homeomorphism. Its inverse $$\mathcal{F}_\mathcal{L}^{-1}: \mathcal{S}(\mathcal{I})\rightarrow H^\infty_\mathcal{L} (\mathcal{F}_\mathcal{L^*}^{-1}: \mathcal{S}(\mathcal{I})\rightarrow H^\infty_\mathcal{L^*})$$ is given by the formula
$$
\left(\mathcal{F}_\mathcal{L}^{-1}h\right)=\sum\limits_{\xi\in \mathcal{I}}h(\xi)e_{\xi},\, h\in \mathcal{S}(\mathcal{I}),$$
$$\left(\left(\mathcal{F}_\mathcal{L^*}^{-1}g\right)=\sum\limits_{\xi\in \mathcal{I}}g(\xi)e^*_{\xi},\, g\in \mathcal{S}(\mathcal{I})\right),
$$
so that the Fourier inversion formula becomes
$$
f=\sum\limits_{\xi\in \mathcal{I}}\hat{f}(\xi)e_\xi\,\,\,\textrm{for all} \,\,\,f\in H^\infty_\mathcal{L},$$
$$\left( h=\sum\limits_{\xi\in \mathcal{I}}\hat{h}_*(\xi)e^*_\xi\,\,\,\textrm{for all} \,\,\,h\in H^\infty_\mathcal{L^*} \right).
$$
Then the Plancherel identity takes the form $$\|f\|_{\mathcal{H}}=\left(\sum\limits_{\xi\in\mathcal{I}}\hat{f}(\xi)\overline{\hat{f}_*(\xi)}\right)^{1/2}.$$
Due to the equivalence we introduce $\mathcal H$--norm as
$$
\|f\|_{\mathcal H}:=\left(\sum\limits_{\xi\in\mathcal{I}} |\widehat{f}(\xi)|^2\right)^{1/2}.
$$

Now we discuss an application of the $\mathcal  L$--Fourier Analysis. Consider a linear operator $L:{\mathcal H}^{\infty}_{\mathcal L}\to
{\mathcal H}^{\infty}_{\mathcal L}$. Define its symbol by
$$
e_{\xi} \sigma_{L}(\xi):= L e_\xi.
$$
Then we obtain
\begin{equation}\label{EQ:T-op}
Lf=\sum_{\xi\in\mathcal{I}} \sigma_{L}(\xi) \, \widehat{f}(\xi) \, e_\xi.
\end{equation}
Note that the correspondence between symbols and operators is one-to-one. We refer to \cite{RT16} for the detailed analysis of symbols and further symbolic calculus.

In general, we can consider the case when $L:{\mathcal H}^{\infty}_{\mathcal L}\to {\mathcal H}^{-\infty}_{\mathcal L}$.

\smallskip

In particular, we have $\sigma_{\mathcal L}(\xi)=\lambda_{\xi}$.

We define Sobolev spaces $\mathcal H^s_{\mathcal L}$ generated by the operator $\mathcal L$ as
\begin{equation}\label{EQ:HsL}
\mathcal H^s_{\mathcal L}:=\left\{ f\in{\mathcal H}^{-\infty}_{\mathcal L}: {\mathcal L}^{s/2}f\in
\mathcal H\right\},
\end{equation}
for any $s\in\mathbb R$ with the norm $\|f\|_{\mathcal H^s_{\mathcal L}}:=\|{\mathcal L}^{s/2}f\|_{\mathcal H}$. Also, we can understand it as
\begin{equation*}\label{EQ:Hsub-norm}
\|f\|_{\mathcal H^s_{\mathcal L}}:=\|{\mathcal L}^{s/2}f\|_{\mathcal H}:=
\left(\sum_{\xi\in\mathcal{I}} |\sigma_{\mathcal L}(\xi)|^{s} |\widehat{f}(\xi)|^{2}\right)^{1/2}.
\end{equation*}

\section{Time-fractional multi-term diffusion-wave equation for self-adjoint operators}
\label{SEC:opL}

Let $D$ 
be a domain of the space variable $x$. 
In a cylindrical domain $\Omega=\{(t,x):\,(0,T)\times D\}$ we consider the equation
\begin{equation}\label{1}
\partial^{\alpha}_{+0,t}u(t,x)-\sum\limits_{j=1}^m a_j\partial^{\alpha_j}_{+0,t}u(t,x)+\mathcal{L} u(t,x)=f(t,x),
\end{equation}
where $\partial^{\alpha}_{+0,t}$ is a Caputo fractional derivative, $0<\alpha_j\leq 1,\, \alpha_j\leq \alpha\leq 2,\, a_j\in \mathbb{R},\, m\in \mathbb{N},$ $f$ is a given function and $\mathcal{L}$ be a linear self-adjoint operator with a discrete spectrum $\{\lambda_\xi: \xi\in \mathcal{I}\}$ on the separable Hilbert space $\mathcal{H}.$ Denote by $e_\xi$ an eigenfunction corresponding to $\lambda_\xi$ of the operator $\mathcal{L}$. Here, $\mathcal{I}$ is a countable set.

A non-local problem for the equation \eqref{1} is formulated as follows:\\
\begin{problem}\label{Pr-1}
To find solutions $u(t,x)$ of the equation \eqref{1} in $\Omega,$ which satisfy
\begin{equation}\label{2}
u(0,x)-\sum\limits_{i=1}^n\mu_i u(T_i,x)=0,\, [\alpha]u_t(0,x)=0,\,x\in \bar{D},
\end{equation}
where $\mu_i\in \mathbb{R},\,0<T_1\leq T_2\leq ...\leq T_n=T.$

A solution of Problem \ref{Pr-1} is the function $u(t, x),$ where $u\in C_{-1}^{2}([0,T]; \mathcal{H})$ and $\mathcal L u\in C_{-1}([0,T]; \mathcal{H})$.
\end{problem}

\begin{assumption}
\label{A1}
In our further results, we will make the following assumptions. Namely, we suppose that
\begin{equation}\label{E1*}
\left|1-\sum\limits_{i=1}^n\mu_i \theta_\xi(T_i)\right|\geq M>0
\end{equation}
hold for all $\xi\in \mathcal{I}$, where $M$ is some positive constant, and $n,\,\mu_i,\,T_i$ are as in \eqref{2}. Here
\begin{align*}
\theta_\xi(t)=E_{(\alpha-\alpha_1,...,\alpha-\alpha_m, \alpha),1}\left(a_1t^{\alpha-\alpha_1},...,a_mt^{\alpha-\alpha_m}, -\lambda_\xi t^\alpha\right),
\end{align*}
and
\begin{multline*}
E_{(\alpha_1,...,\alpha_{m+1}),\beta}(z_1,...,z_{m+1})\\
=\sum\limits_{k=0}^\infty\sum\limits_{
\begin{array}{l}l_1+l_2+...+l_{m+1}=k,\\
l_1\geq 0,...,l_{m+1}\geq 0
\end{array}
}
\frac{k!}{l_1!...l_{m+1}!}\frac{\prod\limits_{j=1}^{m+1} z_j^{l_j}}{\Gamma\left(\beta+\sum\limits_{j=1}^{m+1}\alpha_jl_j\right)}
\end{multline*}
is the multivariate Mittag-Leffler function \cite{LG99}.

In particular, the condition \eqref{E1*} makes sense when
\begin{equation}\label{E1*-1}
\sum\limits_{i=1}^n\left|\mu_i\right| \left|\theta_\xi(T_i)\right|\leq C<1,
\end{equation}
hold for all $\xi\in \mathcal{I}$, for some $C>0$. 
Due to the estimate (see, \cite{LLY15})
$$
\left|\theta_\xi(T_i)\right|\leq\frac{1}{1+\lambda_\xi T_i}
$$
as $\xi\to\infty$, for all $i$, the condition \eqref{E1*-1} needs to be checked only for finite number of $\xi$'s. And it somehow simplifies the situation.

We note that in the cases when $\alpha=1$ and $\alpha=2$ the condition \eqref{E1*-1} can be replaces by
\begin{equation}\label{E1*-2}
\sum\limits_{i=1}^n\left|\mu_i\right| <1.
\end{equation}
The last condition does not depend on $\xi$.
\end{assumption}

Now as an illustration we give several examples of the settings where our nonlocal problems are applicable. Of course, there are many other examples, here we collect the ones for which different types of partial differential equations have particular importance.

$\mathcal{L}$ in \eqref{1} could be any operator from the following list of examples equipped with the corresponding boundary conditions (for its domain).

\begin{itemize}
  \item {\bf Sturm-Liouville problem.}
\end{itemize}
First, we describe the setting of the Sturm-Liouville operator. Let $l$ be an ordinary second order differential operator in $L^2(a,b)$ generated by the differential expression
\begin{equation}\label{SL} l(u)=-u''(x),\,\,a<x<b,\end{equation} and boundary conditions \begin{equation}\label{SL_B} \alpha_1u'(b)+\beta_1u(b)=0,\,\alpha_2u'(a)+\beta_2u(a)=0,\end{equation} where $\alpha_1^2+\alpha_2^2>0,\,\beta_1^2+\beta_2^2>0,$ and $\alpha_j,\, \beta_j,\,j=1,2,$ are some real numbers.

It is known \cite{3} that the Sturm-Liouville problem for \eqref{SL} with boundary conditions \eqref{SL_B} is self-adjoint in $L^2(a,b).$ It is known that the self-adjoint problem has real eigenvalues and their eigenfunctions form a complete orthonormal basis in $L^2(a,b).$
\begin{itemize}
  \item {\bf Differential operator with involution.}
\end{itemize}
As a second example, we consider the second order nonlocal differential operator in $L^2(0,\pi)$ generated by the differential expression
\begin{equation}\label{DOI} l(u)=u''(x)-\varepsilon u''(\pi-x),\,\,0<x<\pi,\end{equation} and boundary conditions \begin{equation}\label{DOI_B} u(0)=0,\,u(\pi)=0,\end{equation} where $|\varepsilon|<1$ some real number.

It is easy to see that an operator generated by \eqref{DOI}-\eqref{DOI_B} is self-adjoint (see. \cite{ToTa17, KST17, AKT17}). For $|\varepsilon|<1,$ the nonlocal problem \eqref{DOI}-\eqref{DOI_B} has the following eigenvalues
\begin{align*}\lambda_{2k}=4(1+\varepsilon)k^2,\,k\in \mathbb{N},\,\,\, \textrm{and} \,\,\, \lambda_{2k+1}=(1-\varepsilon)(2k+1)^2,\,k\in \mathbb{N}\cup\{0\},\end{align*}
and the corresponding system of eigenfunctions
\begin{align*}&u_{2k}(x)=\sqrt{\frac{2}{\pi}}\sin{2kx},\,k\in \mathbb{N},\\& u_{2k+1}(x)=\sqrt{\frac{2}{\pi}}\sin{(2k+1)x},\,k\in \mathbb{N}\cup \{0\}.\end{align*}

\begin{itemize}
  \item {\bf Bessel operator.}
\end{itemize}
In the third example, consider Bessel operator for the expression
\begin{equation}\label{Bessel1}u''(x)+\frac{1}{x}u'(x)-\frac{\nu^2}{x^2}u(x),\,x\in(0,1),\end{equation}
and the boundary conditions
\begin{equation}\label{Bessel2}\lim\limits_{x\rightarrow 0}xu'(x)=0,\,u(1)=0,\end{equation} where $\nu>0.$

The Bessel operator \eqref{Bessel1}-\eqref{Bessel2} is self-adjoint in $L^2(0,1)$ (\cite{KMR18}). It has real eigenvalues $\lambda_k\simeq k\pi+\frac{\nu\pi}{2}-\frac{\pi}{4},\, k\in \mathbb{N},$ and its system of eigenfunctions $\{\sqrt{x}J_\nu(\lambda x)\}_{k\in \mathbb{N}}$ is complete and orthogonal in $L^2(0,1).$

\begin{itemize}
  \item {\bf Fractional Sturm-Liouville operator.}
\end{itemize}
We consider the operator generated by the integro-differential expression
\begin{equation}\label{FSL}\ell (u)=\partial_{+a}^\alpha D_{b-}^{\alpha}u,\,a<x<b,\end{equation}
and the conditions \begin{equation}\label{FSL_B}I_{b-}^{1-\alpha}u(a)=0,\,I_{b-}^{1-\alpha}u(b)=0,\end{equation} where
$\partial_{+a}^\alpha$ is the left Caputo derivative (see. Def. \ref{def3}) of order $\alpha  \in \left( {0,1} \right]$ of $u,$ $$D_{b-}^\alpha  u\left( {x} \right) = -\frac{1}{{\Gamma \left(
{1 - \alpha } \right)}}\frac{d}{dx}\int\limits_x^b {\left( {\xi-x} \right)^{ -\alpha } u\left( \xi \right)}d\xi$$ is the right Riemann-Liouville derivative of order $\alpha \in
\left( {0,1} \right]$ of $u,$ and $$I_{b-}^\alpha  u\left( {x} \right) = \frac{1}{{\Gamma \left({\alpha } \right)}}\int\limits_x^b {\left( {\xi-x} \right)^{\alpha-1} u\left( \xi \right)}d\xi$$ is the right Riemann-Liouville integral of order $\alpha \in
\left( {0,1} \right]$ of $u,$ \cite{KST06}.
The fractional Sturm-Liouville operator \eqref{FSL}-\eqref{FSL_B} is self-adjoint and positive in $L^2 (a, b)$ (see. \cite{TT16, TT18a, TT18b}). The spectrum of the fractional Sturm-Liouville operator generated by the equations \eqref{FSL}-\eqref{FSL_B} is
discrete, positive and real valued, and the system of eigenfunctions is a complete orthogonal basis in $L^2 (a, b).$
\begin{itemize}
  \item {\bf Harmonic oscillator.}
\end{itemize}
\label{Ex-04}
For any dimension $d\geq1$, let us consider the harmonic oscillator,
$$
\L:=-\Delta+|x|^{2}, \,\,\, x\in\mathbb R^{d}.
$$
Then $\L$ is an essentially self-adjoint operator on $C_{0}^{\infty}(\mathbb R^{d})$. It has a discrete spectrum, consisting of the eigenvalues
$$
\lambda_{k}=\sum_{j=1}^{d}(2k_{j}+1), \,\,\, k=(k_{1}, \cdots, k_{d})\in\mathbb N^{d},
$$
and with the corresponding eigenfunctions
$$
\varphi_{k}(x)=\prod_{j=1}^{d}P_{k_{j}}(x_{j}){\rm e}^{-\frac{|x|^{2}}{2}},
$$
which are an orthogonal basis in $L^{2}(\mathbb R^{d})$. We denote by $P_{l}(\cdot)$ the $l$--th order Hermite polynomial, and
$$
P_{l}(\xi)=a_{l}{\rm e}^{\frac{|\xi|^{2}}{2}}\left(x-\frac{d}{d\xi}\right)^{l}{\rm e}^{-\frac{|\xi|^{2}}{2}},
$$
where $\xi\in\mathbb R$, and
$$
a_{l}=2^{-l/2}(l!)^{-1/2}\pi^{-1/4}.
$$
For more information, see for example \cite{NiRo:10}.

\begin{itemize}
  \item {\bf Anharmonic oscillator.}
\end{itemize}
\label{Ex-05}

Another class of examples -- anharmonic oscillators (see for instance
\cite{HR82}), operators on $L^2(\mathbb R)$ of the form
$$
\L:=-\frac{d^{2k}}{dx^{2k}} +x^{2l}+p(x), \,\,\, x\in\mathbb R,
$$
for integers $k,l\geq 1$ and with $p(x)$ being a polynomial of degree $\leq 2l-1$ with real coefficients.

\begin{itemize}
  \item {\bf Landau Hamiltonian in 2D.}
\end{itemize}
\label{Ex-06}
The next example is one of the simplest and most interesting models of the Quantum Mechanics, that is, the Landau Hamiltonian.

The Landau Hamiltonian in 2D is given by
\begin{equation} \label{eq:LandauHamiltonian}
\L:=\frac{1}{2} \p{\p{i\frac{\partial}{\partial x}-B
y}^{2}+\p{i\frac{\partial}{\partial y}+B x}^{2}},
\end{equation}
acting on the Hilbert space $L^{2}(\mathbb R^{2})$, where $B>0$ is some constant. The spectrum of
$\L$ consists of infinite number of eigenvalues (see \cite{F28, L30}) with
infinite multiplicity of the form
\begin{equation} \label{eq:HamiltonianEigenvalues}
\lambda_{n}=\p{2n+1}B, \,\,\, n=0, 1, 2, \dots \,,
\end{equation}
and the corresponding system of eigenfunctions (see \cite{ABGM15, HH13}) is
{\small
\begin{equation*}
\label{eq:HamiltonianBasis} \left\{
\begin{split}
e^{1}_{k, n}(x,y)&=\sqrt{\frac{n!}{(n-k)!}}B^{\frac{k+1}{2}}\exp\Big(-\frac{B(x^{2}+y^{2})}{2}\Big)(x+iy)^{k}L_{n}^{(k)}(B(x^{2}+y^{2})), \,\,\, 0\leq k, {}\\
e^{2}_{j, n}(x,y)&=\sqrt{\frac{j!}{(j+n)!}}B^{\frac{n-1}{2}}\exp\Big(-\frac{B(x^{2}+y^{2})}{2}\Big)(x-iy)^{n}L_{j}^{(n)}(B(x^{2}+y^{2})), \,\,\, 0\leq j,
\end{split}
\right.
\end{equation*}}
where $L_{n}^{(\alpha)}$ are the Laguerre polynomials given by
$$
L^{(\alpha)}_{n}(t)=\sum_{k=0}^{n}(-1)^{k}C_{n+\alpha}^{n-k}\frac{t^{k}}{k!}, \,\,\, \alpha>-1.
$$
Note that in \cite{RT17b, RT18, RT18b} the wave equation for the Landau Hamiltonian with a singular magnetic field is studied.

\subsection{Well-posedness of Problem \ref{Pr-1}}

In this subsection we give and prove the main results of this paper. The condition \eqref{E1*} below can be interpreted as a multi--point non-resonance condition.

\begin{theorem}\label{th1}
Let $f\in C_{-1}([0,T]; \mathcal{H})$ if $\alpha=1$ or $\alpha=2$, and $f\in C_{-1}^{1}([0,T]; \mathcal{H})$ if otherwise. Suppose that Assumption \ref{A1} holds. Then there exists a unique solution of Problem \ref{Pr-1}, and it can be written as
\begin{equation}
\label{E2}
u(t,x)=\sum\limits_{\xi\in\mathcal{I}}\left[F_\xi(t)+\frac{\sum\limits_{i=1}^n\mu_i F_\xi(T_i)}{1-\sum\limits_{i=1}^n\mu_i\theta_\xi(T_i)}\theta_\xi(t)\right]e_\xi(x),
\end{equation}
where
\begin{align*}
&F_\xi(t)=\int\limits_0^t s^{\alpha-1} E_{(\alpha-\alpha_1,...,\alpha-\alpha_m, \alpha),\alpha}\left(a_1s^{\alpha-\alpha_1},...,a_ms^{\alpha-\alpha_m}, -\lambda_\xi s^\alpha\right)f_\xi(t-s)ds,\\& \theta_\xi(t)=E_{(\alpha-\alpha_1,...,\alpha-\alpha_m, \alpha),1}\left(a_1t^{\alpha-\alpha_1},...,a_mt^{\alpha-\alpha_m}, -\lambda_\xi t^\alpha\right).
\end{align*}
Here
\begin{equation*}
f_\xi(t)=\left(f(t,\cdot),e_\xi\right)_{\mathcal{H}},
\end{equation*}
\begin{multline}\label{E3}
E_{(\alpha_1,...,\alpha_{m+1}),\beta}(z_1,...,z_{m+1})\\
=\sum\limits_{k=0}^\infty\sum\limits_{\begin{array}{l}l_1+l_2+...+l_{m+1}=k,\\
l_1\geq 0,...,l_{m+1}\geq 0\end{array}}\frac{k!}{l_1!...l_{m+1}!}\frac{\prod\limits_{j=1}^{m+1} z_j^{l_j}}{\Gamma\left(\beta+\sum\limits_{j=1}^{m+1}\alpha_jl_j\right)}
\end{multline}
is the multivariate Mittag-Leffler function \cite{LG99}.
\end{theorem}

\subsubsection{Proof of the existence result.}
We give a full proof of Problem \ref{Pr-1}. Now we seek a generalised solution by $\mathcal L$--Fourier method. As it was discussed, let $\{\lambda_\xi\}_{\xi\in\mathcal I}$ be the system of eigenvalues and $\{e_\xi(x)\}_{\xi\in\mathcal I}$ be the system of eigenfunctions for the operator $\mathcal{L}$. Since the system of eigenfunctions $e_\xi(x)$ is an orthonormal basis in $\mathcal{H},$ the functions $u(t,x)$ and $f(t,x)$ can be expanded in $\mathcal{H}$ as
\begin{equation}
\label{3}
u(t,x)=\sum\limits_{\xi\in \mathcal{I}}u_\xi(t)e_\xi(x),
\end{equation}
\begin{equation}
\label{4}
f(t,x)=\sum\limits_{\xi\in \mathcal{I}}f_\xi(t)e_\xi(x),
\end{equation}
where $u_\xi(t)$ is unknown and
\begin{equation}
\label{5}
f_\xi(t)=\left(f(t, \cdot), e_\xi\right)_{\mathcal{H}}.
\end{equation}
Substituting functions \eqref{3} and \eqref{4} into the equation \eqref{1}, we obtain the following equations for the unknown functions $u_\xi(t)$:
\begin{equation}
\label{6}
\partial^{\alpha}_{+0}u_\xi(t)-\sum\limits^m_{j=1}a_j\partial^{\alpha_j}_{+0}u_\xi(t)+\lambda_\xi u_\xi(t)=f_\xi(t),
\end{equation}
for all $\xi\in\mathcal I$.

According to \cite{LG99, KMR18}, the solutions of the equations \eqref{6} satisfying initial conditions
\begin{equation}
\label{6*}
u_\xi(0)-\sum\limits_{i=1}^n\mu_i u_\xi(T_i)=0,\,[\alpha]u'_\xi(0)=0,
\end{equation}
can be represented in the form
\begin{equation}
\label{7}
u_\xi(t)=F_\xi(t)+\frac{\sum\limits_{i=1}^n\mu_iF_\xi(T_i)}{1-\sum\limits_{i=1}^n\mu_i \theta_\xi(T_i)}\theta_\xi(t),
\end{equation}
for all $\xi\in\mathcal I$.

We note-that the above expression is well-defined in view of the non-resonance conditions \eqref{E1*}. Finally, based on \eqref{7}, we rewrite our formal solution as \eqref{E2}.

\subsubsection{Convergence of the formal solution.}
Here, we prove convergence of the obtained infinite series corresponding to functions $u(t, x)$, $\partial_{+0,t}^\alpha u(t, x)$, and $\mathcal L u(t, x)$. To prove the convergence of these series, we use the estimate for the multivariate Mittag-Leffler function \eqref{E3}, obtained in \cite{LLY15}, of the form
\begin{align*}
\left|E_{(\alpha-\alpha_1,...,\alpha-\alpha_m, \alpha),\beta}\left(z_1,...,z_{m+1}\right)\right|\leq \frac{C}{1+|z_1|}.
\end{align*}
Let us first prove the convergence of the series \eqref{E2}. From the above estimate, for the functions $F_\xi(t)$ and $\theta_\xi(t),$ we obtain the following inequalities
\begin{align*}
&\left|F_\xi(t)\right|\leq \frac{C\left|\left(f(t, \cdot),e_\xi\right)_{\mathcal{H}}\right|}{1+\lambda_\xi},\\&\left|\theta_\xi(t)\right|\leq \frac{C}{1+\lambda_\xi t^\alpha},\,\,\,C=const>0.
\end{align*}
Hence, from these estimates it follows that
\begin{align*}
|u_\xi(t)|&\leq \left|F_\xi(t)\right|+\frac{\sum\limits_{i=1}^n|\mu_i||F_\xi(T_i)|}{1-\sum\limits_{i=1}^n|\mu_i| |\theta_\xi(T_i)|}|\theta_\xi(t)|\\
& \leq C\frac{\left|\left(f(t, \cdot),e_\xi\right)_{\mathcal{H}}\right|}{1+\lambda_\xi}+\frac{C}{1-M} \sum\limits_{i=1}^n|\mu_i|\frac{\left|\left(f(T_i, \cdot), e_\xi\right)_{\mathcal{H}}\right|}{1+\lambda_\xi}\frac{1}{1+\lambda_\xi t^\alpha}.
\end{align*}
For any fixed $t\in[0, T]$, this implies
\begin{align*}
\|u(t,\cdot)\|_{\mathcal{H}}^{2}& = \|\sum\limits_{\xi\in \mathcal{I}}u_\xi(t)e_\xi\|_{\mathcal{H}}^{2} \\
&= \sum\limits_{\xi\in \mathcal{I}}|u_\xi(t)|^{2} \\
&\leq C \sum\limits_{\xi\in \mathcal I}\left[\frac{\left|\left(f(t, \cdot), e_\xi\right)_{\mathcal{H}}\right|^{2}}{(1+\lambda_\xi)^{2}}+ \frac{\sum\limits_{i=1}^n\left|\left(f(T_i, \cdot), e_\xi\right)_{\mathcal{H}}\right|^{2}}{(1+\lambda_\xi)^{2}}\right]\\
&\leq C \sum\limits_{\xi\in \mathcal I}\left|\left(f(t, \cdot), e_\xi\right)_{\mathcal{H}}\right|^{2}
\end{align*}
and
\begin{equation*}
\begin{split}
\|\mathcal L &u(t, \cdot)\|_{\mathcal{H}}^{2}\\
&\leq C \sum\limits_{\xi\in \mathcal I}\left[\frac{\left|\lambda_\xi\right|^{2} \left|\left(f(t, \cdot), e_\xi\right)_{\mathcal{H}}\right|^{2}}{(1+\lambda_\xi)^{2}}+ \frac{\sum\limits_{i=1}^n\left|\lambda_\xi\right|^{2} \left|\left(f(T_i, \cdot), e_\xi\right)_{\mathcal{H}}\right|^{2}}{(1+\lambda_\xi)^{2}}\right].
\end{split}
\end{equation*}

Since $f(t, \cdot)\in \mathcal{H}$, the series above converge, and we obtain
$$
\|u(t, \cdot)\|_{\mathcal{H}}<\infty
$$
and
$$
\|\mathcal L u(t, \cdot)\|_{\mathcal{H}}<\infty,
$$
for all $t\in[0, T]$.

The convergence of the series corresponding to $u(\cdot, x)$, $\mathcal L u(\cdot, x)$, and $\partial_{+0,t}^\alpha u(\cdot, x)$ for almost all fixed $x$ follows from \cite[Theorem 4.1]{LG99}.

\subsubsection{Proof of the uniqueness result.}
Suppose that there are two solutions $u_1(t,x)$ and $u_2(t,x)$ of Problem \ref{Pr-1}. We denote $$u(t,x)=u_1(t,x)-u_2(t,x).$$ Then the function $u(t,x)$ satisfies the equation \eqref{1} and nonlocal-initial conditions \eqref{2}.

Consider the function
\begin{equation}
\label{U1}
u_\xi(t)=\left(u(t,\cdot), e_\xi\right)_{\mathcal{H}},\,\,\xi\in\mathcal{I}.
\end{equation}
Applying the operator $\left(\partial^{\alpha}_{+0}-\sum\limits^m_{j=1}a_j\partial^{\alpha_j}_{+0}\right)$ to \eqref{U1} from homogeneous equation \eqref{1}, we obtain
\begin{align*}
\partial^{\alpha}_{+0}u_\xi(t)-\sum\limits^m_{j=1}a_j\partial^{\alpha_j}_{+0}u_\xi(t)&=\left(\partial^{\alpha}_{+0,t}u(t,\cdot)
-\sum\limits^m_{j=1}a_j\partial^{\alpha_j}_{+0,t}u(t,\cdot), e_\xi\right)_{\mathcal{H}}\\
& =\left(u(t,\cdot), \mathcal{L} e_\xi\right)_{\mathcal{H}}=\lambda_\xi u_\xi(t).
\end{align*}

Thus, the function \eqref{U1} is the solution of the homogeneous equation \eqref{6} with conditions \eqref{6*}. It is known \cite{LG99} that, the solution of equation \eqref{6} satisfying Cauchy conditions
$$
u_\xi(0) = \rho,\,\rho=const, \,\, [\alpha]u'_\xi(0)=0
$$
can be represented in the form
\begin{align}\label{s1*}u_\xi(t)=\rho E_{(\alpha-\alpha_1,...,\alpha-\alpha_m,\alpha),1} \left(a_1t^{\alpha-\alpha_1},...,a_mt^{\alpha-\alpha_m},-\lambda_\xi t^\alpha\right).\end{align}
Considering $u_\xi(0) = \rho,$ from the first condition in \eqref{6*}, we have
$$\rho-\sum\limits_{i=1}^n\mu_i u_\xi(T_i)=0.$$ Then from \eqref{s1*} it follows
$$\rho\left(1-\sum\limits_{i=1}^n\mu_i E_{(\alpha-\alpha_1,...,\alpha-\alpha_m, \alpha),1}\left(a_1T_i^{\alpha-\alpha_1},...,a_mT_i^{\alpha-\alpha_m},-\lambda_\xi T_i^\alpha\right)\right)=0.$$
If the multi-point conditions \eqref{E1*} hold for all $\xi\in \mathcal I,$ then we obtain
\begin{equation*}
u_\xi(t)=\left(u(t,\cdot), e_\xi\right)_{\mathcal{H}}=0,
\end{equation*}
for all $\xi\in\mathcal{I}$.

Further, by the completeness of the system $\{e_\xi(x)\}_{\xi\in\mathcal I}$ in $\mathcal{H},$ we obtain $$u(t,x)\equiv 0\,\,\, \textrm{for all}\,\,\,t\geq 0,\, x\in \bar{D}.$$ Uniqueness of the solution of Problem \ref{Pr-1} is proved, completing the proof of Theorem \ref{th1}.

\subsection{Non--uniqueness of the solution}
\begin{remark}
\label{Rem1}
If for some $\xi\in\mathcal{I}_1\subset\mathcal{I}$ the expression \eqref{E1*} equals to zero, then the homogeneous Problem \ref{Pr-1} has a nonzero solution.
\end{remark}

Indeed let $a_j=0,\, j=1,...,m$ and $n=1$ in Problem \ref{Pr-1}. Then we reformulate Problem \ref{Pr-1}: {\it To find solutions $u(t, x)$ of the equation
\begin{equation}
\label{1*}
\partial^{\alpha}_{+0,t}u(t, x)+\mathcal{L} u(t, x)=f(t, x),
\end{equation} in $\Omega,$ which satisfy
\begin{equation}\label{2*}
u(0, x)-\mu u(T, x)=0,\, [\alpha]u_t(0, x)=0,\,x\in \bar{D},
\end{equation}
where $\mu\in \mathbb{R},\,1<\alpha\leq 2.$}

If $$|\mu| |E_{\alpha,1}(-\lambda_\xi T^\alpha)|=1$$ for some $\xi\in \mathcal{I}_1\subset \mathcal{I},$ then the homogeneous problem \eqref{1*}-\eqref{2*} has a nonzero solution
\begin{equation}\label{NZS1}
u(t, x)=\sum\limits_{\xi\in\mathcal{I}_1\subset \mathcal{I}}b_\xi E_{\alpha,1}(-\lambda_\xi t^\alpha) e_\xi(x),
\end{equation}
where $b_\xi$ is any real number, and $E_{\alpha,1}\left(z\right)$ is the Mittag-Leffler function
$$E_{\alpha,1}\left(z\right)=\sum\limits_{k=0}^{\infty}
\frac{z^k}{\Gamma\left(\alpha k+1\right)}.$$
Indeed, since the system of eigenfunctions $e_\xi(x)$ is an orthonormal basis in $\mathcal{H},$ the function $u(t,x)$ can be expanded in $\mathcal{H}$ as \eqref{3}
\begin{equation*}
u(t,x)=\sum\limits_{\xi\in \mathcal{I}}u_\xi(t)e_\xi(x),
\end{equation*} where $u_\xi(t)$ is the solution of the equation
\begin{equation}
\label{6**}
\partial^{\alpha}_{+0}u_\xi(t)+\lambda_\xi u_\xi(t)=0,
\end{equation}
with non-local conditions \begin{equation}
\label{6***}
u_\xi(0)-\mu u_\xi(T)=0,\,[\alpha]u'_\xi(0)=0,
\end{equation}
for all $\xi\in \mathcal{I}$.

The general solution of the equation \eqref{6**} is
$$u_\xi(t)=\rho_1E_{\alpha,1}(-\lambda_\xi t^\alpha)+\rho_1tE_{\alpha,2}(-\lambda_\xi t^\alpha),$$ where $\rho_1$ and $\rho_2$ are arbitrary real numbers. Then from \eqref{6***} we obtain
$$
\rho\left(1-\mu E_{\alpha,1}(-\lambda_\xi t^\alpha)\right)=0, \,\,\, \xi\in \mathcal{I}.
$$
If for some $\xi\in \mathcal{I}_1\subset \mathcal{I}$ the condition $|\mu| |E_{\alpha,1}(-\lambda_\xi T^\alpha)|=1$ is true, then the problem \eqref{6**}-\eqref{6***} has a nonzero solution
\begin{equation*}
u_\xi(t)=\rho_\xi E_{\alpha,1}(-\lambda_\xi t^\alpha).
\end{equation*}
Finally, we have \eqref{NZS1}.

The statement of Remark \ref{Rem1} can be extended in the following way.
\begin{theorem}\label{th2}
Let $f\in C_{-1}([0,T]; \mathcal{H})$ if $\alpha=1$ or $\alpha=2$, and $f\in C_{-1}^{1}([0,T]; \mathcal{H})$ if otherwise. Assume that the conditions
\begin{equation}
\label{E1**}
\sum\limits_{i=1}^n|\mu_i| |\theta_\xi(T_i)|=1
\end{equation}
hold for some $\xi\in \mathcal{I}_1\subset \mathcal{I}.$ Then the solution of Problem \ref{Pr-1} is not unique. Moreover, a solution exists if and only if the following condition holds
\begin{equation}\label{E6}
\left(f(t, \cdot),e_\xi\right)_{\mathcal{H}}=0,\,\xi\in\mathcal{I}_1\subset \mathcal{I}.
\end{equation}
If there exists a solution of Problem \ref{Pr-1}, it can be represented in the form
\begin{equation*}
u(t, x)=\sum\limits_{\xi\in{\mathcal{I}_1}}C_\xi\theta_\xi(t)e_\xi(x) + \sum\limits_{\xi\in{\mathcal{I}\backslash\mathcal{I}_1}}\left[F_\xi(t)+\frac{\sum\limits_{i=1}^n\mu_i F_\xi(T_i)}{1-\sum\limits_{i=1}^n\mu_i\theta_\xi(T_i)}\theta_\xi(t)\right]e_\xi(x),
\end{equation*} where $C_\xi$ are some real numbers, and
\begin{align*}
F_\xi(t)=\int\limits_0^t s^{\alpha-1} E_{(\alpha-\alpha_1,...,\alpha-\alpha_m, \alpha),\alpha}\left(a_1s^{\alpha-\alpha_1},...,a_ms^{\alpha-\alpha_m}, -\lambda_\xi s^\alpha\right)f_\xi(t-s)ds,
\end{align*}
and
\begin{align*}
\theta_\xi(t)=E_{(\alpha-\alpha_1,...,\alpha-\alpha_m, \alpha),1}\left(a_1t^{\alpha-\alpha_1},...,a_mt^{\alpha-\alpha_m}, -\lambda_\xi t^\alpha\right).
\end{align*}
\end{theorem}
\begin{proof}
From the proof of Theorem \ref{th1} we know that the formal solution of Problem \ref{Pr-1} has the form \eqref{E2}, that is,
\begin{equation*}
u(t,x)=\sum\limits_{\xi\in\mathcal{I}}\left[F_\xi(t)+\frac{\sum\limits_{i=1}^n\mu_i F_\xi(T_i)}{1-\sum\limits_{i=1}^n\mu_i\theta_\xi(T_i)}\theta_\xi(t)\right]e_\xi(x).
\end{equation*}
Thus, from \eqref{E1**} and \eqref{E6} we verify the validity of the statement of the theorem.
\end{proof}

\subsection{Examples}
Here, we provide as examples several special cases of Problem \ref{Pr-1}. We show that in these cases the condition \eqref{E1*} can be simplified.
\subsubsection{Diffusion equation} We assume that $\alpha=1,$ $a_j=0,\,j=\overline{1,m},$ and $\mu_i\in \mathbb{R},\,i=\overline{1,n}.$ Then, instead of Problem \ref{Pr-1}, we have the classical diffusion equation
$$u_t(t, x)+\mathcal{L}u(t, x)=f(t, x),\,(t, x)\in\Omega,$$
with time-nonlocal condition
$$u(0, x)=\sum\limits_{i=1}^n\mu_i u(T_i, x).$$
Then it is easy to show that the solvability conditions for Problem \ref{Pr-1} has the following form
$$
\sum\limits_{i=1}^n|\mu_i|<1.
$$

\subsubsection{Wave equation} We assume that $\alpha=2,$ $a_j=0,\,j=\overline{1,m},$ and $\mu_i\in \mathbb{R},\,i=\overline{1,n}.$ Let $\mathcal{L}$ is a classical Sturm-Liouville operator with Dirichlet conditions on $[0,1].$ Then, instead of Problem \ref{Pr-1}, we have the classical wave equation
$$u_{tt}(t, x)-u_{xx}(t, x)=f(t, x),\,(t, x)\in(0,T)\times(0,1),$$
with Dirichlet conditions $$u(t, 0)=0,\,u(t, 1)=0,\,\,t\in[0,T],$$ and
time-nonlocal conditions $$u(0, x)=\sum\limits_{i=1}^n\mu_i u(T_i, x),\,u_t(0, x)=0,\,x\in[0,1].$$
Then the solvability conditions for Problem \ref{Pr-1} reduces to \eqref{E1*-2}, namely, to the following form
$$
\sum\limits_{i=1}^n|\mu_i|<1.
$$

\subsubsection{Sub-diffusion equation} We assume that $0<\alpha<1,$ $a_j=0,\,j=\overline{1,m},$ and $\mu_i\in \mathbb{R},\,i=\overline{1,n}.$ Let $\mathcal{L}$ is a Dirichlet-Laplacian $\Delta_D$ on $D\subset \mathbb{R}^d,\,d\geq 2.$ Then, instead of Problem \ref{Pr-1}, we have the classical wave equation
$$\partial_{+0,t}^\alpha u(t, x)-\Delta_Du(t, x)=f(t, x),\,(t, x)\in(0,T)\times D,$$
with time-nonlocal conditions $$u(0, x)=\sum\limits_{i=1}^n\mu_i u(T_i, x), \,x\in\bar{D}.$$
Then it is easy to show that the solvability conditions for Problem \ref{Pr-1} reduces \eqref{E1*-2}, namely, to the following form
$$
\sum\limits_{i=1}^n|\mu_i|<1.
$$

\section{Time-fractional diffusion-wave equation for hypoelliptic differential operators}
\label{SEC:hyp}

In this section we discuss the analogue of the considered problems for the case of the hypoelliptic differential operators. According to the Rothschild-Stein lifting theorem (\cite{RS76}) and its further developments, a model case for general hypoelliptic operators on manifolds are the so-called Rockland operators on graded Lie groups. Such operators have a continuous spectrum, but their global symbols defined by the representation theory of the group, have the discrete spectrum. Thus, the formulae established in the previous section become applicable.

\subsection{Graded Lie groups}
Let us start this section by recalling notations and definitions from the book of Folland and Stein \cite{FS-book} (or \cite[Section 3.1]{FR16}).

We say that a Lie algebra $\mathfrak g$ is graded if there exists a vector space decomposition
$$
\mathfrak g=\bigoplus_{j=1}^\infty V_j
$$
such that $[V_i,V_j]\subset V_{i+j}$. Let $\G$ be a connected simply connected Lie group. Then we call $\G$ is a graded Lie group if its Lie algebra $\mathfrak g$ is graded. In the case when the first stratum $V_1$ generates $\mathfrak g$ as an algebra, we say it is a stratified group.

Graded Lie groups are nilpotent and homogeneous with a canonical choice of dilations \cite{RuzSur}. Namely, we define $A$ by setting
$AX=\nu_j X$ for $X\in V_j$. Then the dilations on $\mathfrak g$ can be represented by
$$
D_r:={\rm Exp}(A\ln r), \; r>0.
$$
Then the number
$$
Q:=\nu_1+\ldots+\nu_n={\rm Tr}\, A
$$
is the {\em homogeneous dimension} of $\G$.

Henceforth, let $\G$ be a graded Lie group. In \cite{Rockland}, Rockland gave an original definition of `Rockland' operators using the language of representations.
In \cite{HN-79}, Helffer and Nourrigat showed that a left-invariant differential operator of homogeneous positive degree satisfies the Rockland condition if and only if it is hypoelliptic. Such operators are called {Rockland operators}.
Namely, by following \cite[Definition 4.1.1]{FR16},
we call $\R$ a Rockland operator on the graded Lie group $\G$ if $\R$ is a homogeneous (of an order $\nu\in\mathbb N$) and left-invariant differential operator satisfying the Rockland condition:
\begin{itemize}
\item the operator $\pi(\R)$ is injective on $\mathcal H^{\infty}_{\pi}$, that is,
$$
\pi(\R)v=0 \,\,\, \Rightarrow \,\,\, v=0, \,\,\, \text{for all} \,\,\, v\in \mathcal H^{\infty}_{\pi}
$$
for any representation $\pi\in \widehat{\G}$, excluding the trivial case.
\end{itemize}
Here we denote by $\widehat{\G}$ the unitary dual of $\G$, $\mathcal H^{\infty}_{\pi}$ is the vector space of smooth vectors for $\pi\in\widehat{\G}$,
and $\pi(\R)$ stands for the infinitesimal representation of $\R$ which is an element of the universal enveloping algebra of the garaded Lie group $\G$. The readers are referred to \cite[Chapter 4]{FR16} for more details on graded Lie groups and Rockland operators.

We say that a vector $v$ from the separable Hilbert space $\mathcal{H}_{\pi}$ is a smooth if the function
$$
\G\ni x \mapsto \pi(x)v\in \mathcal{H}_{\pi}
$$
is of class $C^{\infty}$, for a representation $\pi$ of the graded Lie group $\G$ on $\mathcal{H}_{\pi}$.
Denote by $\mathcal{H}_{\pi}^{\infty}$ the space of all smooth vectors of a representation $\pi$.
Let $\pi:\mathcal{H}_{\pi}\to\mathcal{H}_{\pi}$ be a strongly continuous representation of $\G$. We denote
$$
d\pi(X)v:=\lim_{t\rightarrow 0}\frac{1}{t}\left(\pi(\exp_{\G}(tX))v-v\right)
$$
for every $X\in \mathfrak g$ and $v\in \mathcal{H}_{\pi}^{\infty}$.
Then $d\pi$ is a representation of the graded Lie algebra $\mathfrak{g}$ on $\mathcal{H}_{\pi}^{\infty}$, namely, the
infinitesimal representation associated to $\pi$, for example, see \cite[Proposition 1.7.3]{FR16}).
Here, we will often write $\pi$ instead of $d\pi$ and $\pi(X)$ instead of $d\pi(X)$ for all $X\in\mathfrak{g}$.

By the Poincar\'{e}-Birkhoff-Witt theorem, any left-invariant differential operator $T$ on the graded Lie group $\G$ has the form
\begin{equation}\label{PBW_for}
T=\sum_{|\alpha|\leq M}c_{\alpha}X^{\alpha}.
\end{equation}
Let $\mathfrak{U}(\mathfrak g)$ be the universal enveloping algebra of $\mathfrak g$. Then the form \eqref{PBW_for} allows us to write $T\in\mathfrak{U}(\mathfrak g)$.
Note that in \eqref{PBW_for} all but finitely many of $c_{\alpha}\in \mathbb{C}$ are equal to zero and $X^{\alpha}=X_{1}\cdots X_{|\alpha|},$ with $X_{j}\in \mathfrak{g}$.
Thus, the family of infinitesimal representations $\{\pi(T),\pi \in \widehat{\G}\}$ gives a field of operators, that is, the symbol associated with $T$.

Let $\R$ be a positive homogeneous Rockland operator of degree $\nu>0$. Then, for $\pi\in \widehat{\G}$ from \eqref{PBW_for} we have the infinitesimal representation of $\R$ related to $\pi$, namely,
$$
\pi(\R)=\sum_{[\alpha]=\nu}c_{\alpha}\pi(X)^{\alpha},
$$
where $[\alpha]=\nu_{1}\alpha_{1}+\cdots+\nu_{n}\alpha_{n}$ and $\pi(X)^{\alpha}=\pi(X^{\alpha})=\pi(X_{1}^{\alpha_{1}}\cdots X_{n}^{\alpha_{n}})$. Here $[\alpha]$ is the homogeneous degree of $\alpha$.

From now on we assume that $\R$ and $\pi(\R)$ are self-adjoint operators acting on $L^{2}(\G)$ and $\mathcal{H}_{\pi}$ with dense domains $\mathcal{D}(\G)\subset L^{2}(\G)$ and $\mathcal{H}_{\pi}^{\infty}\subset \mathcal{H}_{\pi}$, respectively (see e.g. \cite[Proposition 4.1.15]{FR16}).

For the spectral measures $E$ and $E_{\pi}$ corresponding to $\R$ and $\pi(\R)$, by the spectral theorem \cite[Theorem VIII.6]{RS80}, we have
$$
\R=\int_{\mathbb{R}}\lambda dE(\lambda)$$ and $$\pi(\R)=\int_{\mathbb{R}}\lambda dE_{\pi}(\lambda).
$$
From here we can observe that the representations $\pi(\R)$ of a positive self--adjoint Rockland operator $\R$ are also positive.

Note that for an arbitrary measurable bounded function $\phi$ on $\mathbb{R}$ we obtain
\begin{equation}\label{prel_for_1}
\mathcal{F} (\phi(\R)f)(\pi)=\phi(\pi(\R))\widehat{f}(\pi),
\end{equation}
for all $f\in L^{2}(\G)$ (see e.g. \cite[Corollary 4.1.16]{FR16}).
The fact that the spectrum of the representation $\pi(\R)$ ($\pi \in \widehat{\G}\backslash \{1\}$) is discrete and lies in $\mathbb R_{>0}$ is showed by Hulanicki, Jenkins and Ludwig in \cite{HJL85}.
Namely, it allows one to choose an orthonormal basis of $\mathcal{H}_{\pi}$ such that the infinite matrix associated to the self-adjoint operator $\pi(\R)$ ($\pi \in \widehat{\G}\backslash \{1\}$) is
\begin{equation}\label{pi_R_matrix}
\pi(\R)=\begin{pmatrix} \pi_{1}^{2} & 0 & \ldots &\ldots \\
0 & \pi_{2}^{2} & 0 & \ldots\\
\vdots & 0 & \ddots &\\
\vdots & \vdots &  & \ddots \end{pmatrix},
\end{equation}
where $\pi_{j}\in (0, +\infty)$.

Now we briefly recall some elements of the Fourier analysis on $\G$. We start by defining the group Fourier transform of $f$ at $\pi$, that is,
$$
\mathcal{F}_{\G}f(\pi)\equiv \widehat{f}(\pi)\equiv \pi(f):=\int_{\G}f(x)\pi(x)^{*}dx,
$$
for $f\in L^{1}(\G)$ and $\pi \in \widehat{G}$, with integration with respect to the biinvariant Haar measure on the graded Lie group $\G$.
It implies a linear mapping $\widehat{f}(\pi): \mathcal{H}_{\pi}\to\mathcal{H}_{\pi}$ that can be represented by an infinite matrix when we choose a basis for $\mathcal{H}_{\pi}$. Thus, we obtain
$$
\mathcal{F}_{\G}(\R f)(\pi)=\pi(\R)\widehat{f}(\pi).
$$
In what follows, we will use the same basis (for $\mathcal{H}_{\pi}$) as in \eqref{pi_R_matrix} when we write $\widehat{f}(\pi)_{m,k}$.

By using Kirillov's orbit method one can construct explicitly the Plancherel measure $\mu$ on the dual space $\widehat{\G},$ see c.g. \cite{CG90}. In the particular case, we obtain the Fourier inversion formula
\begin{equation}\label{EQ: Group-FT-g}
f(x)=\int_{\widehat{\G}}\mathrm{Tr}[\widehat{f}(\pi)\pi(x)]d\mu(\pi),
\end{equation}
where $\mathrm{Tr}$ is the trace operator, and $d\mu(\pi)$ is the Plancherel measure on $\widehat{\G}$.
Moreover, $\pi(f)=\widehat{f}(\pi)$ is the Hilbert-Schmidt operator, that is,
$$
\|\pi(f)\|^{2}_{{\rm HS}}={\rm Tr}(\pi(f)\pi(f)^{*})<\infty,
$$
and $\widehat{\G}\ni \pi \mapsto \|\pi(f)\|^{2}_{{\rm HS}}$ is an integrable function with respect to $\mu$. Furthermore, the following Plancherel formula holds
\begin{equation}\label{planch_for}
\int_{\G}|f(x)|^{2}dx=\int_{\widehat{\G}}\|\pi(f)\|^{2}_{{\rm HS}}d\mu(\pi),
\end{equation}
see e.g. \cite{CG90} or \cite{FR16}.

We refer to \cite{FR16} for more details of the Fourier analysis on the graded Lie groups.

Thus, in our extension of the obtained result to graded Lie groups, we will work with positive Rockland operators $\R.$ To give some examples,
this setting includes:
\begin{itemize}
  \item {\bf Homogeneous elliptic differential operators.}
\end{itemize}
Let $\mathbb{G}=\mathbb{R}^n.$ Then a Rockland operator $\R$ may be any positive homogeneous partial differential operator of elliptic type, with constant coefficients. For example, $m$-Laplacian operator
$$\R=(-\Delta)^m$$ or higher order elliptic operator $$\R=(-1)^m\sum\limits_{j=1}^n a_j\left(\frac{\partial}{\partial x_j}\right)^{2m},\,\,\, a_j>0,\,\,m\in \mathbb{N}.$$

\begin{itemize}
  \item {\bf Hypoelliptic operators on the Heisenberg group.}
\end{itemize}
If $\mathbb{G}=\mathbb{H}^n$ is a Heisenberg group. Then, we can take $$\R=(-\mathcal{L}_{\mathbb{H}^n})^m,\,\,\,m\in \mathbb{N},$$ where $\mathcal{L}_{\mathbb{H}^n}$ is a sub-Laplacian on the Heisenberg group $\mathbb{H}^n.$
\begin{itemize}
  \item {\bf Hypoelliptic operators on Carnot groups.}
\end{itemize}
Let $\mathbb{G}$ be a Carnot group (or a stratified group) with vectors $X_1,...,X_s$ spanning the first stratum. Then, a Rockland operator $\R$ can be given by
$$\R=(-1)^m\sum\limits_{j=1}^sa_jX_j^{2m},\,\,\,a_j>0,\,\,m\in \mathbb{N}.$$ In particular case, if $m=1,$ then $\R$ is a sub-Laplacian operator.

\begin{problem}\label{R: Pr-1} Let $\G$ be a graded Lie group with a homogeneous dimension $d\geq 3$ and let $\R$ be a positive self--adjoint Rockland operator acting on $L^{2}(\G)$. Assume that Let $f\in C_{-1}([0,T]; L^2(\G))$ if $\alpha=1$ or $\alpha=2$, and $f\in C_{-1}^{1}([0,T]; L^2(\G))$ if $0<\alpha<1$ or $1<\alpha<2$. In a domain $\Omega=\{(t,x):\,(0,T)\times \G\}$ consider the equation
\begin{equation}\label{R: 1}
\partial^{\alpha}_{+0,t}u(t,x)-\sum\limits_{j=1}^m a_j\partial^{\alpha_j}_{+0,t}u(t,x)+\R u(t,x)=f(t,x),
\end{equation}
with non-local initial conditions
\begin{equation}\label{R: 2}
u(0,x)-\sum\limits_{i=1}^n\mu_i u(T_i,x)=0,\, [\alpha]u_t(0,x)=0,\,x\in \G,
\end{equation}
where $\partial^{\alpha}_{+0,t}$ is a Caputo fractional derivative, $0<\alpha_j\leq 1,\, \alpha_j\leq \alpha\leq 2,\, a_j\in \mathbb{R},\, m\in \mathbb{N}$, and $\mu_i\in \mathbb{R},\,0<T_1\leq T_2\leq ...\leq T_n=T.$

We seek a solution $u\in C_{-1}^{2}([0,T]; L^{2}(\G))$ of Problem \eqref{R: 1}--\eqref{R: 2} such that $\R u\in C_{-1}([0,T]; L^{2}(\G))$.
\end{problem}

The condition \eqref{R: E1*} below can be interpreted as a multi--point non-resonance condition.

\begin{theorem}\label{R: th1}
Let $f\in C_{-1}([0,T]; L^2(\G))$ if $\alpha=1$ or $\alpha=2$, and $f\in C_{-1}^{1}([0,T]; L^2(\G))$ if otherwise. Assume that the conditions
\begin{equation}\label{R: E1*}
\sum\limits_{i=1}^n|\mu_i| |\widehat {\theta}(t,\pi)_{l}|\leq M<1,\, M>0
\end{equation}
hold for all $l\in \mathbb N$ (where $M$ is a constant), where $n,\,\mu_i,\,T_i$ are as in \eqref{2}. Then there exists a unique solution of Problem \ref{R: Pr-1}, and it can be written as
\begin{equation}
\label{R: E2}
u(t,x)=\int_{\widehat{\G}}\mathrm{Tr}[\widehat{K}(t, \pi)\pi(x)]d\mu(\pi),
\end{equation}
where
$$
\widehat{K}(t, \pi)_{l, k}=\widehat {F}(t,\pi)_{l,k}+\frac{\sum\limits_{i=1}^n\mu_i \widehat {F}(T_i,\pi)_{l,k}}{1-\sum\limits_{i=1}^n\mu_i \widehat{\theta}(T_i,\pi)_{l}} \widehat {\theta}(t,\pi)_{l},
$$
for all $\pi \in \widehat{\G}$ and $l,k\in \mathbb{N}$. Here
\begin{align*}
&\widehat {F}(t,\pi)_{l,k}=\int\limits_0^t s^{\alpha-1} E_{(\alpha-\alpha_1,...,\alpha-\alpha_m, \alpha),\alpha}\left(a_1s^{\alpha-\alpha_1},...,a_m s^{\alpha-\alpha_m}, -\pi_l^2 s^\alpha\right)\widehat{f}(t-s,\pi)_{l,k}ds,\\
&\widehat {\theta}(t,\pi)_{l}=E_{(\alpha-\alpha_1,...,\alpha-\alpha_m, \alpha),1}\left(a_1t^{\alpha-\alpha_1},...,a_m t^{\alpha-\alpha_m}, -\pi_l^2 t^\alpha\right).
\end{align*}
\end{theorem}

\subsection{Proof of Theorem \ref{R: th1}}
\subsubsection{Proof of the existence result.}
We give a full proof of Problem \ref{R: Pr-1}.
Let us take the group Fourier
transform of \eqref{R: 1} with respect to $x\in\G$ for all $\pi\in\widehat{G}$, that is,
\begin{equation}\label{fou_cauchy1}
\partial^{\alpha}_{+0,t}\widehat {u}(t,\pi)-\sum\limits_{j=1}^m a_j\partial^{\alpha_j}_{+0,t}\widehat {u}(t,\pi)+\pi(\R)\widehat{u}(t,\pi)=\widehat{f}(t,\pi).
\end{equation}
Taking into account \eqref{pi_R_matrix}, we rewrite the matrix equation \eqref{fou_cauchy1} componentwise as an infinite system of equations of the form
\begin{equation}\label{R: 6}
\partial^{\alpha}_{+0,t}\widehat {u}(t,\pi)_{l,k}-\sum\limits_{j=1}^m a_j\partial^{\alpha_j}_{+0,t}\widehat {u}(t,\pi)_{l,k}+\pi_{l}^{2}\widehat{u}(t,\pi)_{l,k}=\widehat{f}(t,\pi)_{l,k},
\end{equation}
for all $\pi \in \widehat{\G}$, and any $l, k\in \mathbb{N}$. Now let us decouple the system given by the matrix equation \eqref{fou_cauchy1}. For this, we fix an arbitrary representation $\pi$, and a general entry $(l, k)$ and we treat each equation given by \eqref{R: 6} individually.

According to \cite{LG99, KMR18}, the solutions of the equations \eqref{R: 6} satisfying initial conditions
\begin{equation}
\label{R: 6*}
\widehat {u}(0,\pi)_{l,k}-\sum\limits_{i=1}^n\mu_i \widehat {u}(T_i,\pi)_{l,k}=0,\,[\alpha] \partial_{t}\widehat {u}(0,\pi)_{l,k}=0,
\end{equation}
can be represented in the form
\begin{equation}
\label{R: 7}
\widehat {u}(t,\pi)_{l,k} = \widehat {F}(t,\pi)_{l,k}+\frac{\sum\limits_{i=1}^n\mu_i \widehat {F}(T_i,\pi)_{l,k}}{1-\sum\limits_{i=1}^n\mu_i \widehat{\theta}(T_i,\pi)_{l}} \widehat {\theta}(t,\pi)_{l},
\end{equation}
for all $\pi \in \widehat{\G}$ and any $l,k\in \mathbb{N}$, where
\begin{align*}
&\widehat {F}(t,\pi)_{l,k}=\int\limits_0^t s^{\alpha-1} E_{(\alpha-\alpha_1,...,\alpha-\alpha_m, \alpha),\alpha}\left(a_1s^{\alpha-\alpha_1},...,a_m s^{\alpha-\alpha_m}, -\pi_l^2 s^\alpha\right)\widehat{f}(t-s,\pi)_{l,k}ds,\\
&\widehat {\theta}(t,\pi)_{l}=E_{(\alpha-\alpha_1,...,\alpha-\alpha_m, \alpha),1}\left(a_1t^{\alpha-\alpha_1},...,a_m t^{\alpha-\alpha_m}, -\pi_l^2 t^\alpha\right).
\end{align*}

Then there exists a solution of Problem \ref{R: Pr-1}, and it can be written as
\begin{equation}
\label{R: E2}
u(t,x)=\int_{\widehat{\G}}\mathrm{Tr}[\widehat{K}(t, \pi)\pi(x)]d\mu(\pi),
\end{equation}
where
$$
\widehat{K}(t, \pi)_{l, k}=\widehat {F}(t,\pi)_{l,k}+\frac{\sum\limits_{i=1}^n\mu_i \widehat {F}(T_i,\pi)_{l,k}}{1-\sum\limits_{i=1}^n\mu_i \widehat{\theta}(T_i,\pi)_{l}} \widehat {\theta}(t,\pi)_{l},
$$
for all $\pi \in \widehat{\G}$ and $l,k\in \mathbb{N}$.

We note-that the above expression is well-defined in view of the non-resonance conditions \eqref{R: E1*}. Finally, based on \eqref{R: 7}, we rewrite our formal solution as \eqref{R: E2}.

\subsubsection{Convergence of the formal solution.}
Here, we prove convergence of the obtained infinite series corresponding to functions $u(t, x)$, $\partial_{+0,t}^\alpha u(t, x)$, and $\R u(t, x)$. To prove the convergence of these series, we use the estimate for the multivariate Mittag-Leffler function \eqref{E3}, obtained in \cite{LLY15}, of the form
\begin{align*}
\left|E_{(\alpha-\alpha_1,...,\alpha-\alpha_m, \alpha),\beta}\left(z_1,...,z_{m+1}\right)\right|\leq \frac{C}{1+|z_1|}.
\end{align*}
Let us first prove the convergence of the series \eqref{R: E2}. From the above estimate, for the functions $F_\xi(t)$ and $\theta_\xi(t),$ we obtain the following inequalities
\begin{align*}
&\left|\widehat {F}(t,\pi)_{l,k}\right|\leq C \frac{\left|\widehat{f}(t,\pi)_{l,k}\right|}{1+\pi_l^2},\,\left|\widehat {\theta}(t,\pi)_{l}\right|\leq \frac{C}{1+\pi_l^2t^\alpha},\,\,\,\,C=const>0.
\end{align*}
Hence, from these estimates it follows that
\begin{align*}
|\widehat {u}(t,\pi)_{l,k}|&\leq \left|\widehat {F}(t,\pi)_{l,k}\right|+\sum\limits_{i=1}^n|\mu_i||\widehat {F}(T_i,\pi)_{l,k}||\widehat {\theta}(t,\pi)_{l}|\\
& \leq C \frac{\left|\widehat{f}(t,\pi)_{l,k}\right|}{1+\pi_l^2} + C \sum\limits_{i=1}^n|\mu_i| \frac{\left|\widehat{f}(T_i,\pi)_{l,k}\right|}{1+\pi_l^2} \frac{1}{1+\pi_l^2t^\alpha}.
\end{align*}

Thus, since for any Hilbert-Schmidt operator $A$ one has
$$\|A\|^{2}_{{\rm HS}}=\sum_{l,k}|(A\phi_{l},\phi_{k})|^{2}$$
for any orthonormal basis $\{\phi_{1},\phi_{2},\ldots\}$, then we can consider the infinite sum over $l,k$ of the inequalities provided by \eqref{R: 7}, we have
\begin{equation}
\label{eq4}
\|\widehat {u}(t,\pi)\|^{2}_{{\rm HS}} \leq C \|(1+\pi(\R))^{-1})\widehat{f}(t,\pi)\|^{2}_{{\rm HS}}.
\end{equation}
Thus, integrating both sides of \eqref{eq4} against the Plancherel measure $\mu$ on $\widehat{\G}$, then using the Plancherel identity \eqref{planch_for} we obtain
\begin{align*}
\|u(t, \cdot)\|_{L^{2}(\G)} \leq C \|(I+\R)^{-1}f(t, \cdot)\|_{L^{2}(\G)}
\end{align*}
and
\begin{align*}
\|\R u(t, \cdot)\|_{L^{2}(\G)} \leq C \|f(t, \cdot)\|_{L^{2}(\G)},
\end{align*}
for any fixed $t\in[0, T]$.

Since $f(t, \cdot)\in L^{2}(\G)$, the series above converge, and we obtain
$$
\|u(t, \cdot)\|_{L^{2}(\G)}<\infty\,\,\,\textrm{and}\,\,\,\|\R u(t, \cdot)\|_{L^{2}(\G)}<\infty,
$$
for all $t\in[0, T]$.

The convergence of the series corresponding to $u(\cdot, x)$, $\R u(\cdot, x)$, and $\partial_{+0,t}^\alpha u(\cdot, x)$ for almost all fixed $x\in\G$ follows from \cite[Theorem 4.1]{LG99}.

The uniqueness result can be proved in analogy to the previous arguments.

\section*{Acknowledgements}
The authors were supported in parts by the FWO Odysseus Project 1 grant G.0H94.18N: Analysis and Partial Differential Equations. The first author was supported in parts by the EPSRC grant EP/R003025/1 and by the Leverhulme Grant RPG-2017-151. The second author was supported by the Ministry of Education and Science of the Republic of Kazakhstan Grant AP05130994. The third author was supported by Ministry of Education and Science of the Republic of Kazakhstan Grant AP05131756. No new data was collected or generated during the course of research.





 \it

 \noindent
$^1$ Department of Mathematics: \\ Analysis,
Logic and Discrete Mathematics\\
Ghent University, Krijgslaan 281, \\ Building S8
B 9000 Ghent, Belgium\\[2pt]
$^2$ School of Mathematical Sciences\\
Queen Mary University of London\\
London, United Kingdom \\[2pt]
  e-mail: michael.ruzhansky@ugent.be
\hfill Received: December 4, 2018 \\[6pt]
$^3$ Al--Farabi Kazakh National University\\
71 Al--Farabi ave., Almaty, 050040, Kazakhstan\\[2pt]
$^4$ Institute of Mathematics and Mathematical Modeling
125 Pushkin str., Almaty, 050010, Kazakhstan \\[2pt]
  e-mail: niyaz.tokmagambetov@ugent.be\\[2pt]
  e-mail: berikbol.torebek@ugent.be

\end{document}